\documentclass[11pt]{amsart}
\usepackage[dvipsnames]{xcolor}
\usepackage[english]{babel}
\usepackage[T1]{fontenc}
\usepackage[utf8]{inputenc}
\usepackage{amsmath}
\usepackage{amssymb}
\usepackage{amsthm}
\usepackage{color}
\usepackage[a4paper]{geometry}
\usepackage{tikz}
\usepackage[all]{xy}
\usepackage{hyperref}
\usepackage{lmodern}
\usepackage{mathrsfs}
\usetikzlibrary{patterns}
\usetikzlibrary{decorations.pathmorphing}
\usetikzlibrary{decorations.pathreplacing,angles,quotes}
\tikzset{snake it/.style={decorate, decoration=snake}}

\theoremstyle{definition}
\newtheorem{definition}{Definition}[section]

\newtheorem{rmk}[definition]{Remark}

\theoremstyle{plain}
\newtheorem{theorem}[definition]{Theorem}

\newtheorem{cor}[definition]{Corollary}
\newtheorem{lem}[definition]{Lemma}

\newcommand{\real}{\mathbb{R}}

\newcommand{\mb}{\mathbb}
\newcommand{\mc}{\mathcal}
\newcommand{\mf}{\mathfrak}

\newcommand{\bs}{\boldsymbol}

\newcommand{\vol}{\textup{vol}}

\newcommand{\s}{\mathbb{S}}
\newcommand{\Z}{\mathbb{Z}}

\newcommand{\dd}{\textup{d}}

\newcommand{\Ltwo}{\textup{L}^{2}}
\newcommand{\Lone}{\textup{L}^{1}}

\newcommand{\ignore}[1]{}

\newif\ifdraft\drafttrue

\draftfalse

\numberwithin{equation}{section}

\usepackage{scalerel,stackengine}
\stackMath
\newcommand\reallywidehat[1]{%
\savestack{\tmpbox}{\stretchto{%
  \scaleto{%
    \scalerel*[\widthof{\ensuremath{#1}}]{\kern-.6pt\bigwedge\kern-.6pt}%
    {\rule[-\textheight/2]{1ex}{\textheight}}
  }{\textheight}%
}{0.5ex}}%
\stackon[1pt]{#1}{\tmpbox}%
}
\begin{document}

\thanks{This paper was written during an REU project organized by the department of mathematics of the University of Michigan, Ann Arbor, whose support is gratefully acknowledged. The third author is sponsored by the Derksen Research Incentive fund}

\title{Higher-Dimensional Moving Averages and Submanifold Genericity}

\author[Jiajun Cheng]{Jiajun Cheng}
\address{Department of Mathematics, University of California, Riverside CA, USA} 
\email{djcheng@umich.edu}

\author[Reynold Fregoli]{Reynold Fregoli}
\address{{Department of Mathematics, University of Michigan, Ann Arbor MI, USA}} 
\email{{reynldfregoli@gmail.com}}

\author[Beinuo Guo]{Beinuo Guo}
\address{Department of Mathematics, University of Michigan, Ann Arbor MI, USA} 
\email{carlguo@umich.edu}


\begin{abstract}
We generalize results of Jones and Olsen on multi-parameter moving ergodic averages to measure-preserving actions of  $\mb R^d$ for $d\geq 1$. In particular, we give necessary and sufficient conditions for the pointwise convergence of averages over families of boxes in $\mb R^d$. As an application of our characterization, we show that averages along dilates of "locally flat" submanifolds in $\mb R^d$ do not necessarily converge point-wise for bounded measurable functions. This is closely related to the concept of submanifold-genericity recently introduced in \cite{BFK25}.
\end{abstract}

\maketitle

\section{Introduction}

Let $(X,\mu)$ be a non-atomic probability space and let $T_1,\dotsc, T_d$ be  commuting measure-preserving transformations on the space $X$. Given sequences $\{\bs n_k\}\subset \mb Z^d$ and $\{\bs l_k\}\subset \mb N^d$ ($k\in \mb N$), put
$$B_k:=[n_{k1},n_{k1}+l_{k1})\times\dotsm\times[n_{kd},n_{kd}+l_{kd})$$
and let $\mc B:=\{B_k\}$. To any box $B_k\in \mc B$ associate an averaging operator $A_k$ defined by
$$A_kf(x):=\frac{1}{l_{k1}\dotsm l_{kd}}\sum_{\bs j\in B_k\cap\mb Z^d}f\left(T_1^{j_1}\dotsb T_d^{j_d}x\right)$$
for $f\in \Lone(X)$, and denote by ${M_{\mc B}}$ the maximal operator
$${M_{\mc B}}f(x):=\sup_{B_k\in\mc B}\frac{1}{l_{k1}\dotsm l_{kd}}\sum_{\bs j\in B_k\cap\mb Z^d}\left|f\left(T_1^{j_1}\dotsb T_d^{j_d}x\right)\right|.$$
In \cite{BJR90}, Bellow, Jones, and Rosenblatt gave necessary and sufficient conditions on the sequence $\mc B$ for $M_{\mc B}$ to satisfy a maximal inequality in dimension $d =1$. Jones and Olsen \cite{JONES1991127} generalized the result to the case $d \geq 1$.

Let us briefly recall their characterization (see \cite[Theorem 2.1]{JONES1991127}). For each $i=1,\dotsc,d$ let
$$\Omega_{i}:=\{(n_{ki},l_{k,i}):k\in \mb N\}.$$
For fixed $\alpha>0$ and $\lambda\in\mb N$ define
$$\Omega_i^{(\alpha)}:=\left\{(x,y)\in \Z \times \mb N:|x-n|\leq \alpha(y-l)\mbox{ for some }(n,l)\in \Omega_i\right\}$$
and
$$\Omega_i^{(\alpha)}(\lambda):=\left\{x\in \mb Z:(x,\lambda)\in\Omega_i^{(\alpha)}\right\}.$$
Then, according to \cite[Theorem 2.1 Part (a)]{JONES1991127}, $M_{\mc B}$ is of strong type $(p,p)$ for any $p>1$ if and only if the following condition on the sequence $\mc B$ is satisfied for all coordinates $i=1, \ldots, d$
\begin{equation}
\tag{$C_1$}
\label{eq:cone}
\exists\, A,\alpha>0:\quad\#\Omega_i^{(\alpha)}(\lambda)\leq A\lambda\quad\mbox{for all }\lambda\in \mb N.
\end{equation}
Condition \eqref{eq:cone} asserts, in other words, that the horizontal cross-section of all the cones of aperture $\alpha$ above points in the sequence $\Omega_i$ at fixed height must grow linearly as the height increases (see picture below). For example, the sequence $(k,rk)$ for any $r\in \mb N$ has this property, while the sequence $(k,\sqrt{k})$ does not.

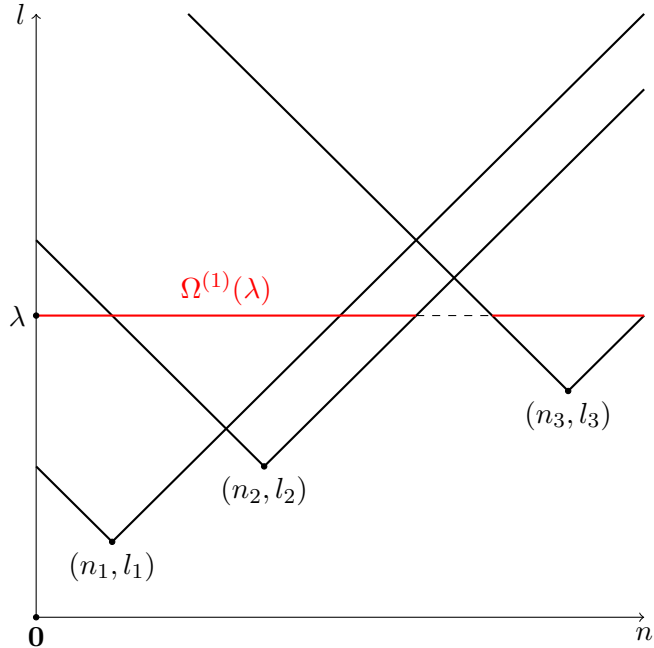
\begin{figure}[!h]
\centering

\begin{tikzpicture}
\draw[->, black, thin] (0,0) -- (0,8); 
\draw[->, black, thin] (0,0) -- (8,0); 

\filldraw[black] (0,0) circle (1pt) node[anchor=north]{$\bs 0$};
\draw[black,thick] (0,8) node[anchor=east]{$l$};
\draw[black,thick] (8,0) node[anchor=north]{$n$};

\filldraw[black] (1,1) circle (1pt) node[anchor=north]{$(n_{1i},l_{1i})$};
\filldraw[black] (3,2) circle (1pt) node[anchor=north]{$(n_{2i},l_{2i})$};
\filldraw[black] (7,3) circle (1pt) node[anchor=north]{$(n_{3i},l_{3i})$};

\draw[black,thick] (1,1) -- (0,2);
\draw[black,thick] (1,1) -- (8,8);

\draw[black,thick] (3,2) -- (0,5);
\draw[black,thick] (3,2) -- (8,7);

\draw[black,thick] (7,3) -- (2,8);
\draw[black,thick] (7,3) -- (8,4);

\draw[red,thick] (0,4) -- (5,4);
\draw[black, dashed, thin] (5,4) -- (6,4);
\draw[red,thick] (6,4) -- (8,4);

\filldraw[black] (0,4) circle (1pt) node[anchor=east]{$\lambda$};

\draw[red,thick] (2.5,4) node[anchor=south]{$\Omega_i^{(1)}(\lambda)$};

\end{tikzpicture}

\caption{Set $\Omega_i^{(1)}(\lambda)$ for some sequence $(n_{ki},l_{ki})$.}

\end{figure}

\begin{rmk}
Note that when $d=1$, \cite{BJR90}[Theorem 1 Part (a)] shows that $M_{\mc B}$ is of weak type $(1,1)$. For $d\geq 2$, this does not follow from the proof in \cite{JONES1991127}. In fact, as discussed in \cite[Remark 1]{JONES1991127}, there exist sequences of boxes $\mc B$ satisfying \eqref{eq:cone} such that $M_{\mc B}$ is not of weak type $(1,1)$.
\end{rmk}

If the operator $M_{\mc B}$ is of strong type $(p,p)$ and the action generated by the transformations $T_1, \ldots, T_d$ is ergodic, it follows from a standard argument the sequence of functions $A_kf$ converges point-wise for any $f\in\mathrm{L}^p$. \cite[Theorem 3.2]{JONES1991127} shows additionally that this is not the case if Condition \eqref{eq:cone} fails in at least one coordinate.

In this paper, we generalize \cite[Theorem 2.1, Part (a)]{JONES1991127} and \cite[Theorem 3.2]{JONES1991127} to measurable actions of $\mb R^d$ for $d\geq 1$. When $d=1$, an alternative proof can be found in \cite{Ferrando_1995}[Section 4] as a special case (sequence of moving cubes of dimension 1). Our main motivation for considering such theorems is to study the property of submanifold genericity, recently introduced in \cite{BFK25}. In particular, based on our extension of Jones and Olsen's result we are able to show that in aperiodic actions of $\mb R^d$ an ergodic theorem for dilates of general submanifolds is not to be expected, unless the test functions are very regular (e.g., smooth). We now proceed to lay out our results in greater detail.

\subsection{The Continuous Case}



Let $(X,\mu)$ be a non-atomic probability space and let $U_{\bs t}$ for $\bs t\in \mb R^d$ denote a measurable and measure-preserving $d$-dimensional flow on $X$ (which is our standing assumption throughout this section). Given sequences $\{\bs w_k\}\subset\mb R^d$ and $\{\bs s_k\}\subset(0,+\infty)^d$, put
$$B_k:=[w_{k1},w_{k1}+s_{k1})\times\dotsm\times[w_{kd},w_{kd}+s_{kd})$$
and let $\mc B:=\{B_k\}$. To any box $B_k\in \mc B$ associate a continuous averaging operator $R_k$ defined by
$$R_kf(x):=\frac{1}{s_{k1}\dotsm s_{kd}}\int_{B_k}f\left(U_{\bs t}x\right)\,\dd\bs t$$
for $f\in \Lone(X)$, and let ${N_{\mc B}}$ denote the maximal operator
$${N_{\mc B}}f(x):=\sup_{B_k\in\mc B}\frac{1}{s_{k1}\dotsm s_{kd}}\int_{B_k}\left|f\left(U_{\bs t}x\right)\right|\,\dd\bs t.$$
For $i=1,\dotsc,d$ let
$$\Omega_{i}:=\{(w_{ki},s_{k,i}):k\in \mb N\}$$
and for $\alpha>0$ define
$$\Omega_{i}^{(\alpha)}:=\{(x,y)\in \mb R\times (0, +\infty):|x-t|\leq \alpha(y-s)\mbox{ for some }(t,s)\in \Omega_i\}.$$
Finally, for any real $\lambda>0$ put
$$\Omega_{i}^{(\alpha)}(\lambda):=\left\{x\in \mb R:(x,\lambda)\in \Omega_i^{(\alpha)}\right\}.$$
Then, in analogy to \cite{JONES1991127}[Theorem 2.1 Part (a)], we have the following.

\begin{theorem}
\label{thm:cont1}
Assume that for all $i=1,\dotsc,d$
\begin{equation}
\label{eq:allomegacont}
\tag{$\tilde C_i$}
\exists\, A,\alpha>0:\quad \textup{Leb}\left(\Omega_{i}^{(\alpha)}(\lambda)\right)\leq A\lambda\quad\mbox{for all }\lambda>0,   
\end{equation}
where $\textup{Leb}$ denotes the Lebesgue measure on $\mb R$. Then the operator ${N_{\mc B}}$ is of strong type $(p,p)$ for any $p>1$.   
\end{theorem}

As a corollary of \ref{thm:cont1}, we show the following.

\begin{cor}
\label{cor:cont1}
Assume that the flow $U_{\bs t}$ acts ergodically on the space $(X,\mu)$ and that Condition \eqref{eq:allomegacont} is satisfied for all $i=1,\dotsc,d$. Then for any $f\in \textup{L}^p(X)$ with $p>1$ and $\mu$-a.e. $x\in X$ we have that
$$R_{k}f(x)\to\mu(f)\quad\mbox{as }k\to\infty.$$
\end{cor}

If Condition \eqref{eq:allomegacont} is not satisfied for at least one $1\leq i\leq d$, on the other hand, and the action of $\mb R^d$ on $X$ is aperiodic, the conclusion of Corollary \ref{cor:cont1} fails in a strong sense (cf. \cite[Theorem 3.2]{JONES1991127}). Recall that a $d$-dimensional flow $U_{\bs t}$ on $X$ is said to be \textit{aperiodic} if there exists a set $N\subset X$, with $\mu(N)=0$, such that $U_{\bs t}x\neq x$ for all $x\in X\setminus N$ and all $\bs t \in \real^d\setminus\{\bs 0 \}$.

Before formally stating the converse of Corollary \ref{cor:cont1}, we recall the definition of \textit{mixing family} introduced by Sawyer. Following \cite{Sawyer}, we say that a family of measure-preserving transformations $\{S_{h}\}$ on $X$ is \textit{mixing} if for any pair of measurable subsets $A, B$ of $X$ and any $\rho>1$
\begin{equation}
\label{eq:Sawmix}
\exists\, h:\quad\mu(A\cap S^{-1}_{h}(B))< \rho \cdot \mu(A)\mu(B).    
\end{equation} 

\begin{rmk}
\label{rmk:ergismix}
Note that the notion of "mixing family" differs from that of "mixing action". More precisely, suppose that $\{S_{h}\}$ is a locally compact group of measure-preserving transformations on $X$. Then Condition \eqref{eq:Sawmix} is equivalent to
$$\inf_{h} \mu\left(A\cap S_{h}(B)\right)\leq \mu(A)\mu(B),$$    
which differs from the hypothesis that in a mixing action
$$\lim_{S_{h}\to\infty}\mu\left(A\cap S_{h}(B)\right)= \mu(A)\mu(B).$$
In fact, it is possible to show that any group of measure-preserving transformations that acts ergodically on $X$ satisfies \eqref{eq:Sawmix} (see, e.g., \cite[Lemma 1, page 163]{Sawyer}). This is an easy exercise if $\{S_{h}\}$ is just the group generated by one ergodic transformation on $X$.
\end{rmk}

We are now in a position to state our next result.

\begin{theorem}
\label{thm:cont2}
Assume the flow $U_{\bs t}$ is aperiodic and that it commutes with a mixing family of transformations $\{S_h\}$ on $X$. Then, if for some $1\leq i\leq d$ it holds that
\begin{equation}
\label{eq:ctslinearfail}
\tag{$!\tilde C_i$}
\forall\,A,\alpha>0\quad \textup{Leb}\left(\Omega_{i}^{(\alpha)}(\lambda)\right)\geq A\cdot\lambda\quad\mbox{for an unbounded set of }\lambda>0, 
\end{equation}
the operators $R_k$ have the "strong sweeping out" property, that is, for any $\varepsilon>0$ there exists a set $B_{\varepsilon}\subset X$, with $\mu(B_{\varepsilon})<\varepsilon$, such that for $\mu$-a.e. $x\in X$ it holds that
$$\liminf_{k}R_k\chi_{B_{\varepsilon}}(x)=0\quad\mbox{and}\quad\limsup_{k}R_k\chi_{B_{\varepsilon}}(x)=1.$$
\end{theorem}

\begin{rmk}
\label{rmk:subgrab}
It is worth pointing out that, in view of Remark \ref{rmk:ergismix}, any subgroup of an Abelian group that acts ergodically on the space $(X,\mu)$ commutes with a mixing family, even if the subgroup itself does not act ergodically.    
\end{rmk}

In the next subsection, we present the main application of Theorem \ref{thm:cont2}.

\subsection{Failure of Submanifold Genericity for Essentially Bounded Functions}

Let $(X,\mu)$ be a probability space equipped with a measure preserving action of $\mb R^d$, which we denote by $\bs a.x$ for all $\bs a\in\mb R^d$ and $x\in X$. In \cite{BFK25}, the notion of submanifold genericity for a measure $\nu$ on $X$ was introduced, in connection with certain problems in Diophantine approximation. We recall it here, for the convenience of the reader.

Let $M\subset\mb R^d$ be a compact $m$-dimensional $\mathscr{C}^1$ submanifold of $\mb R^d$ and let $\mc F\subset \Lone(X)$ be an arbitrary collection of functions. Let $\vol_m$ denote the $m$-dimensional volume measure induced by the Euclidean metric on $\mb R^d$. We say that a measure $\nu$ on $X$ (potentially equal to $\mu$) is $(M,\mc F)$-generic if for $\nu$-almost every $x\in X$ it holds that
\begin{equation}
\label{eq:conv}
\frac{1}{t^m\cdot \textup{vol}_m(M)}\int_{tM}f(\bs a.x)\,\dd\textup{vol}_m( \bs a)\to\mu(f)\quad\mbox{as }t\to\infty.    
\end{equation}

We also say that a measure $\nu$ is $(m,\mc F)$-generic (for fixed $1\leq m\leq d$) if for all compact $m$-dimensional $\mathscr{C}^1$ submanifolds $M$ in $\mb R^d$ it holds that $\nu$ is $(M,\mc F)$-generic. 

\medskip

When $M$ is a $d$-dimensional bounded positive-measure set, $(M, \Lone(X))$-genericity is equivalent to ergodicity of the $\mb R^d$-action (see \cite[Theorem 1.2]{Lindenstrauss}). However, if $M$ is a proper submanifold of $\mb R^d$, the convergence in \eqref{eq:conv} is more delicate and may not occur for all $f\in\Lone(X)$. For example, when $M=\s^{d-1}$ (i.e., the $(d-1)$-dimensional sphere in $\mb R^d$) and $d \geq 3$, Jones \cite[Theorem 2.1]{Jones} showed that $\mu$ is $(M, \textup{L}^{p}(X))$-generic for any $p>d/(d-1)$ and that the bound on $p$ is sharp (see \cite[Theorem 2.3]{Jones} for a constructive counterexample). \cite{Lacey} extended this result to the case $d=2$. Such analytical arguments, however, do not apply to different manifolds $M$, such as (portions of) the boundary of a hypercube.


If one restricts to very special families of functions $\mc F$, on the other hand, there are actions for which point-wise convergence in \eqref{eq:conv} occurs for virtually any sufficiently regular submanifold of $\mb R^d$. For example, let $X=G/\Gamma$, where $G$ is a semisimple Lie group (with no factors of rank $1$) and $\Gamma$ is a lattice in $G$. Let $\mu$ denote the canonical left-invariant measure on $G/\Gamma$ and let $A<G$ be a Cartan subgroup, whose Lie algebra will be denoted by $\mf a$. Consider the action of $\mf a$ on $X$ given by $x\mapsto \exp(a)x$  (which is identifiable with an $\mb R^{d}$ action on $X$, if $G$ is of rank $d$). Then \cite[Theorem 1.3]{BFK25} asserts that the measure $\mu$ on $X$ is $(k,\mathscr{C}^{\infty}_c(X))$-generic for all $1\leq k\leq d$. Here, $\mathscr{C}^{\infty}_c(X)$ stands for the space of smooth compactly supported functions on $X$. In addition, any invariant measure supported on a closed unipotent orbit in $X$ is also $(k,\mathscr{C}^{\infty}_c(X))$-generic for all $k=1,\dotsc,d$ (see for instance \cite[Theorem 1.4]{BFK25}). More generally, any measure-preserving and exponentially mixing action of $\mb R^d$ on a probability space $(X,\mu)$ (under some minor assumptions) enjoys the same property (see \cite[Theorem 2.2]{BFK25}).

In view of the above discussion, it is natural to ask whether there exist ergodic actions of $\mb R^d$ on some probability space $(X,\mu)$ such that the measure $\mu$ is $(k,\mc F)$-generic for some much larger class of functions $\mc F$, e.g., $\mathrm L^{\infty}(X)$.
As a consequence of Theorem \ref{thm:cont2}, we are able to show the following.

\begin{theorem}
\label{thm:cont3}
Let $(X,\mu)$ be a probability space, equipped with an ergodic and aperiodic measure-preserving action of $\mb R^d$. Let $M$ be a compact $m$-dimensional ($1\leq m\leq d-1$) $\mathscr{C}^1$ submanifold of $\mb R^d$ such that $M\cap \pi$ is a non-empty open set in $M$ for a given $m$-dimensional affine subspace $\pi$ of $\mb R^d$, which does not contain the origin. Then the measure $\mu$ is not $(M,\textup{L}^{\infty}(X))$-generic. In particular, there exist no ergodic and aperiodic measure-preserving actions of $\mb R^d$ on $X$ that are $(m,\textup{L}^{\infty}(X))$-generic for any $m=1,\dotsc,d-1$. 
\end{theorem}

\begin{rmk}
\label{rmk:cartisaperiodic}
Note that the left-multiplication action of a Cartan subgroup $A$ of a Lie group $G$ on the quotient $X=G/\Gamma$ is always ergodic and aperiodic (see explanation below), so that both \cite[Theorem 1.3]{BFK25} and Theorem \ref{thm:cont3} apply simultaneously.
In other words, a point-wise ergodic theorem for dilates of submanifolds $M$ such as those in Theorem \ref{thm:cont3} holds for smooth test functions, but fails as soon as the regularity assumption is dropped (e.g., if the test function is just essentially bounded). It is also interesting to observe that for the same submanifolds the convergence in \eqref{eq:conv} still occurs in the $\textup{L}^p$-norm for any $p\geq 1$, i.e., a mean ergodic theorem holds. This follows easily from the fact that the measure $\mu$ is $(m,\mathscr{C}^{\infty}_c(X))$-generic, by a density argument.

Let us briefly explain why the left-multiplication action of a Cartan subgroup $A$ on $G$ is aperiodic. If $ag\Gamma=g\Gamma$ for some $a\in A$ and $g\in G$, then there exists $\gamma\in\Gamma$ such that $g^{-1}ag=\gamma$. Put $S(A,\gamma):=\{g\in G:g\gamma g^{-1}\in A\}$ and observe that any periodic point $g\Gamma$ for the group $A$ belongs to the set
$$\bigcup_{\gamma\in\Gamma}S(A,\gamma)\Gamma.$$
This is a countable union (since $\Gamma$ is a lattice in $G$) of algebraic varieties and therefore has measure $0$.
\end{rmk}

We conclude with the following observation. According to \cite[Theorem 2]{BF09}, if $M$ is a polynomial curve in $\real^d$ not entirely contained in a proper affine subspace,
the convergence in \eqref{eq:conv} occurs in norm (more precisely, in the $\textup{L}^2$ norm). As suggested in \cite{BF09}, it would be natural to study the point-wise convergence in \eqref{eq:conv} for such curves. However, polynomial curves fall outside the scope of Theorem \ref{thm:cont3}, since any analytic submanifold $M$ of $\mb R^d$ cannot intersect an affine subspace in an open set, unless it is entirely contained in the subspace.

\subsection{Acknowledgments}

The authors are indebted to Alexander Gorodnik for suggesting to read \cite{BJR90}, to Dmitry Kleinbock for pointing out Remark \ref{rmk:cartisaperiodic}, and to Amos Nevo for many fruitful conversations which eventually led to this paper.

\section{Proof of Theorem \ref{thm:cont1} and of Corollary \ref{cor:cont1}}

\subsection{Proof of Theorem \ref{thm:cont1}}

First, we observe that, if a box $B_k$ contains both vectors whose $i$-th component is negative and vectors whose $i$-th component is positive for some $k$ and $i$, then it must be $w_{ki}<0$ and $s_{ki}>|w_{ki}|$. In this case, we may write
$$B_{k}=B_k^{i-}\cup B_k^{i+},$$
where
$$B_k^{i-}:=B_k\cap\{x_i<0\}\quad\mbox{and}\quad B_{k}^{i+}=B_k\cap \{x_i\geq 0\}.$$
If $\mc B'$ is the collection of boxes where $B_k$ is replaced by $B_k^{i-}$ and $B_k^{i+}$, it is obvious that ${N_{\mc B}}f\leq N_{\mc B'}f$ for any $f\in \textup{L}^p(X)$. Moreover, for any $k$ and $i$ such that $s_{ki}>|w_{ki}|$ and any $\lambda >0$ we have that
$$\mathrm{Leb}(\{x\in \mb R:|x-w_{ki}|\leq \alpha(\lambda-s_{ki})\})\leq (1+\alpha)\lambda,$$
so that the boxes $B_k$ that are broken in two or more parts do not contribute to the validity of Condition \eqref{eq:allomegacont}. 
Hence, by working in each orthant separately and by replacing $U_{t_i}$ with $U_{t_i}^{-1}$ if necessary, we may always assume that $w_{ki}\geq 0$ for all $k$ and $i$.

Now, the proof for $d>1$ relies on the case $d=1$ (with $s_k>0$). To see this let $U_t^{(i)}:=U_{t\bs e_i}$, where $\bs e_i$ denotes the $i$-th vector of the standard basis in $\mb R^d$. Take $f\in \textup{L}^p(X)$ for fixed $p>1$ and observe that
\begin{align*}
{N_{\mc B}}f(x) & =\sup_{B_k\in\mc B}\frac{1}{s_{k1}\dotsm s_{kd}}\int_{B_k}\left|f\left(U^{(1)}_{t_1}\dotsb U^{(d)}_{t_d}x\right)\right|\,\dd \bs t\\
& \leq \sup_{(w_{k1},s_{k1})\in \Omega_1}\dotsb \sup_{(w_{kd},s_{kd})\in \Omega_d}\frac{1}{s_{k1}\dotsm s_{kd}}\int_{t_1=w_{k1}}^{w_{k1}+s_{k1}}\dotsb \int_{t_d=w_{kd}}^{w_{kd}+s_{kd}}\left|f\left(U^{(1)}_{t_1}\dotsb U^{(d)}_{t_d}x\right)\right|\,\dd\bs t\\
& \leq \sup_{(w_{k1},s_{k1})\in \Omega_1}\frac{1}{s_{k1}}\int_{t_1=w_{k1}}^{w_{k1}+s_{k1}}\left|\ \dotsb\sup_{(w_{kd},s_{kd})\in \Omega_d}\frac{1}{s_{kd}}\int_{t_d=w_{kd}}^{w_{kd}+s_{kd}}\left|f\left(U^{(1)}_{t_1}\dotsb U^{(d)}_{t_d}x\right)\right|\right|\, \dd \bs t\\[2mm]
&=M_{\Omega_1}\circ\dotsb \circ M_{\Omega_d}f(x),
\end{align*}
where
$$M_{\Omega_i}f(x):=\sup_{(w_{ki},s_{ki})\in\Omega_i}\frac{1}{s_{ki}}\int_{t=w_{ki}}^{w_{ki}+s_{ki}}\left|f\left(U^{(i)}_tx\right)\right|\, \dd t$$
for $i=1,\dotsc,d$. Since the functions ${M_{\mc B}}f$ and $M_{\Omega_1}\circ\dotsb\circ M_{\Omega_d}f$ are positive, we deduce that
$$\left\|{M_{\mc B}}f\right\|_p\leq \|M_{\Omega_1}\circ\dotsb\circ M_{\Omega_d}f\|_p.$$
Now, by the case $d=1$, each operator $M_{\Omega_i}$ is of strong type $(p,p)$. Then ${M_{\mc B}}$ is also of strong type $(p,p)$.

For the case $d=1$ we modify the proof of \cite[Theorem 1]{BJR90}. We stress once again that the conclusion for $d=1$ can also be deduced by the arguments in \cite{Ferrando_1995}, however, modifying \cite[Theorem 1]{BJR90} gives a much simpler proof, which we present to keep the paper self-contained. 

Recall that a family of intervals $\mc B=\{[w_k,w_k+s_k)\}$ with $w_k\geq 0$ and $s_k>0$ is given and that this family satisfies Condition \eqref{eq:allomegacont}. Our goal is to show that the operator
$$N_{\mc B}f(x):=\sup_{k}\frac{1}{s_k}\int_{0}^{s_k}f(U_{w_k+t}x)\,\dd t$$
is of strong type $(p,p)$ for any $p >1$. We will in fact show that it is of weak type $(1,1)$ and conclude by a classical interpolation argument.

In what follows, for any $p\geq 1$ and any measure space $Y$ we denote by $\mathscr{L}^p(Y)$ the set of $p$-integrable functions from $Y$ to $\mb R$. For any $f\in \mathscr{L}^1(X)$, $T>0$, $x\in X$, and $t>0$ we also put
$$f_{T,x}(t):=f(U_tx)\cdot\chi_{[-T,T]}(t).$$
Then, clearly, $f_{T,x}\in \mathscr{L}^1(\mb R)$. Denote by $f_{T,x}^{*}$ the one-sided Hardy-Littlewood maximal function associated to $f_{T,x}$, that is
$$f_{T,x}^{*}(\xi):=\sup_{r\neq 0}\frac{1}{|r|}\int_{0}^{r}|f(\xi+\tau)|\,\dd \tau,$$
where the bounds of the integral are to be inverted if $r<0$. Then, by the Hardy-Littlewood maximal inequality \cite[Lemma 1.6.16]{Tao11}, we have that for any $\lambda>0$
\begin{equation}
\label{eq:HL}
\textup{Leb}\left(\xi:f_{T,x}^{*}(\xi)>\lambda\right)\leq\frac{\|f_{T,x}\|_1}{\lambda}=\frac{1}{\lambda}\int_{-T}^{T}\left|f(U_t x)\right|\,\dd t.    
\end{equation}

Let us now fix $\lambda>0$ and a function $f\in \mathscr{L}^1(X)$. Recall that we aim to show that
\begin{equation}
\label{eq:th}
\mu\left(x:{N_{\mc B}}f(x)>\lambda\right)\leq \frac{\textup{const.}}{\lambda}\|f\|_1.    
\end{equation}

First, observe that
$$\left\{x:{N_{\mc B}}f(x)>\lambda\right\}=\bigcup_{k}\{x:R_k|f|(x)>\lambda\}.$$
Let $\varepsilon>0$ and let $K_{\varepsilon}\geq 1$ such that
\begin{equation*}
\mu\left(\bigcup_{k\leq K_{\varepsilon}}\{x:R_k|f|(x)>\lambda\}\right)\\
\geq \mu\left(x:{N_{\mc B}}f(x)>\lambda\right)-\varepsilon.
\end{equation*}
Choose $N_{\varepsilon}$ so large that
$$\frac{\max_{k\leq K_{\varepsilon}} \left\{ w_{k}+s_k\right\}}{N_{\varepsilon}}\leq\varepsilon$$
and set
$$T_{\varepsilon}:=N_{\varepsilon}+\max_{k\leq K_{\varepsilon}} \left\{w_{k}+s_k\right\}.$$
Further, define
$$Y_{\varepsilon}:=\bigcup_{k\leq K_{\varepsilon}}\{(x,t):R_k|f|(U_tx)>\lambda\mbox{ and }|t|<N_{\varepsilon}\}$$
and observe that, by the invariance of $\mu$, we have that
\begin{multline}
\label{eq:yeps}
\mu\otimes\textup{Leb}(Y_{\varepsilon})=\int_{-N_{\varepsilon}}^{N_{\varepsilon}}\int_{X}\chi_{\{\sup_{k\leq K_{\varepsilon}}R_k|f|>\lambda\}}\circ U_t\,\dd\mu\dd t \\
2N_{\varepsilon}\cdot \mu\left(\bigcup_{k\leq K_{\varepsilon}}\{x:R_k|f|(x)>\lambda\}\right)\geq 2N_{\varepsilon}\cdot\mu\left(x:{N_{\mc B}}f(x)>\lambda\right)-2N_{\varepsilon}\cdot \varepsilon.
\end{multline}
Our goal will now be to bound from above the measure of the set $Y_{\varepsilon}$. To do so, we consider sections
$$Y_{\varepsilon}(x):=\{\xi:(x,\xi)\in Y_{\varepsilon}\}$$
for $x\in X$. In the next two lemmas, we obtain a bound for $\textup{Leb}(Y_{\varepsilon}(x))$ in terms of the quantity
$$\textup{Leb}\left(\xi:f_{T,x}^{*}(\xi)>\lambda\right)$$
for fixed $x$. Thanks to \eqref{eq:HL} and \eqref{eq:yeps}, this will easily lead to \eqref{eq:th} (details explained in the sequel).

\begin{lem}
\label{lem:lemma1}
For any $(x,t)\in Y_{\varepsilon}$ there exists a pair $(w_k,s_k)\in\Omega$ with $k\leq K_{\varepsilon}$ such that
\begin{equation}
\label{eq:wholeintervalinside}
[t+w_k,t+w_k+s_k)\subset \left\{\xi:f_{T_{\varepsilon},x}^{*}(\xi)>\lambda\right\}.    
\end{equation}
\end{lem}

\begin{proof}
Fix $(x,t)\in Y_{\varepsilon}$. Then there exists $(w_k,s_k)\in\Omega$ with $k\leq K_{\varepsilon}$ such that
\begin{equation}
\label{eq:geqlambda}
\frac{1}{s_k}\int_{0}^{s_k} |f(U_{t+w_k+\xi}x)|\,\dd \xi=\frac{1}{s_k}\int_{0}^{s_k} |f_{T_{\varepsilon},x}(t+w_k+\xi)|\,\dd \xi>\lambda.
\end{equation}
Assume by contradiction that \eqref{eq:wholeintervalinside} fails. Then for some $\xi_0\in [t+w_k,t+w_k+s_k)$ we have that $f_{T_{\varepsilon},x}^{*}(\xi_0)\leq \lambda$. Hence,
\begin{multline*}
\int_{-(\xi_0-(t+w_k))}^{0}|f_{T_{\varepsilon},x}(\xi_0+\tau)|\,\dd \tau\leq \lambda\cdot (\xi_0-(t+w_k))\quad\mbox{and} \\
\int_{0}^{(t+w_k+s_k)-\xi_0}|f_{T_\varepsilon,x}(\xi_0+\tau)|\,\dd\tau\leq \lambda\cdot ((t+w_k+s_k)-\xi_0)
\end{multline*}
and 
$$\int_{t+w_k}^{t+w_k+s_k}|f_{T_{\varepsilon},x}(\tau)|\,\dd\tau\leq\lambda\cdot s_k$$
-- a contradiction to \eqref{eq:geqlambda}.
\end{proof}

Lemma \ref{lem:lemma1} further implies the following.

\begin{lem}
\label{lem:lemma2}
For any $x\in X$ it holds that
$$\textup{Leb}\left(Y_{\varepsilon}(x)\right)\leq AC_{\alpha}\cdot \textup{Leb}\left(\xi:f_{T_{\varepsilon},x}^{*}(\xi)>\lambda\right),$$
where $A$ and $\alpha$ are the constants in Condition \eqref{eq:allomegacont} (for $i=1$) and $C_{\alpha}$ is an absolute constant only depending on $\alpha$.
\end{lem}

\begin{proof}
Fix $\delta>0$ and $x\in X$ such that $Y_{\varepsilon}(x)\neq\emptyset$. Note that, by construction $Y_{\varepsilon}(x)\subset(-N_{\varepsilon},N_{\varepsilon})$. Let $O_x\subset (-N_{\varepsilon},N_{\varepsilon})$ be an open set with the following two properties:
\begin{equation}
\label{eq:contup}
O_x\supseteq Y_{\varepsilon}(x)\quad\mbox{and}\quad\mu\left(O_x\setminus Y_{\varepsilon}(x)\right)\leq \delta.    
\end{equation}
For each $\xi\in O_x$ put $r_\xi:=\sup\{r>0:(\xi,\xi+r)\subset O_x\}$ and consider the covering of $O_x$ given by
$$\bigcup_{\xi\in O_x}(\xi,\xi+r_\xi).$$
By the Vitali Covering Lemma (see \cite[Theorem 6.2.1]{Cohn13}), there exists a countable sub-collection of \emph{disjoint} open intervals $(\xi_i,\xi_i+r_i)$ (where $r_i=r_{\xi_i}$ for brevity) such that
$$O_{x}\subseteq\bigcup_i(\xi_i-2r_i,\xi_i+3r_i).$$
Fix $t\in Y_{\varepsilon}(x)$. By Lemma \ref{lem:lemma1}, we may find $k\leq K_{\varepsilon}$ such that
$$[t+w_k,t+w_k+s_k)\subset O_x.$$
Now, if there exists $i$ such that
$$t+w_k+\frac{s_k}{3}<\xi_i<t+w_k+\frac{2s_k}{3},$$
by the definition of $r_\xi$ it must be $r_i\geq s_k/3$. On the other hand, if for all $i$ if holds that
$$\xi_i<t+w_k+\frac{s_k}{3}\quad\mbox{or}\quad \xi_i\geq t+w_k+\frac{2s_k}{3},$$ 
then, there must be $i$ such that $t+w_k+s_k/2\in (\xi_i-2r_i,\xi_i+3r_i)$. Hence,
$3r_i\geq s_k/6$. In any case, we have that there exists $i$ such that
$r_i\geq s_k/18$ and $(\xi_i-2r_i,\xi_i+3r_i)\cap [t+w_k,t+w_k+s_k)\neq\emptyset$. For this $i$ we therefore have that
$$(\xi_i-21r_i,\xi_i+21r_i)\supseteq [t+w_k,t+w_k+s_k).$$ 
Now, note that
$$|(\xi_i-21r_i)-t-w_k|\leq |(\xi_i-21r_i)-(t+w_k)|\leq 42r_i\leq \alpha(C_{\alpha}r_i-s_k)$$
for $C_{\alpha}=42\alpha^{-1}+18$. By definition of $\Omega^{(\alpha)}$, we conclude that
$$\xi_i-21r_i-t\in\Omega^{(\alpha)}(C_{\alpha}r_i).$$
Hence, we have that
$$Y_{\varepsilon}(x)\subset \bigcup_{i}(\xi_i-21 r_i)-\Omega^{(\alpha)}(C_{\alpha}r_i).$$
This shows that
\begin{multline*}
\textup{Leb}\left(Y_{\varepsilon}(x)\right)\leq \sum_{i}\textup{Leb}\left(\Omega^{(\alpha)}(C_{\alpha}r_i)\right)\leq \sum_{i}AC_{\alpha}r_i \\
\leq AC_{\alpha}\cdot \textup{Leb}(O_x)\leq AC_{\alpha}\cdot\left(\textup{Leb}\left(\xi:f_{T,x}^{*}(\xi)>\lambda\right)+\delta\right).\end{multline*}
By letting $\delta\to 0$, we conclude.
\end{proof}

On combining \eqref{eq:HL} and Lemma \ref{lem:lemma2}, we find that for any $x\in X$ it holds that
\begin{equation}
\label{eq:leb}
\textup{Leb}\left(Y_{\varepsilon}(x)\right)\leq AC_{\alpha}\cdot\textup{Leb}\left(\xi:f_{T_{\varepsilon},x}^{*}(\xi)>\lambda\right)\leq \frac{AC_{\alpha}}{\lambda}\cdot\int_{-T_{\varepsilon}}^{T_{\varepsilon}}|f(U_tx)|\,\dd t.    
\end{equation}

We now use transference. From \eqref{eq:yeps} and \eqref{eq:leb}, it follows that
\begin{multline*}
2N_{\varepsilon}\cdot \mu(x:{N_{\mc B}}f(x)>\lambda)- 2N_{\varepsilon}\cdot\varepsilon \\
\leq \mu\otimes\textup{Leb}(Y_{\varepsilon})
=\int_{X}\textup{Leb}(Y_{\varepsilon}(x))\dd\mu(x) \\
\leq \frac{AC_{\alpha}}{\lambda}\int_{X}\int_{-T_{\varepsilon}}^{T_{\varepsilon}}|f(U_tx)|\,\dd t\dd\mu\leq 2T_{\varepsilon}\cdot \frac{AC_{\alpha}}{\lambda}\|f\|_1.    
\end{multline*} 
Dividing both sides by $2N_{\varepsilon}$ (recall that $1\leq T_{\varepsilon}/N_{\varepsilon}\leq 1+\varepsilon$) gives
$$\mu\left(x:{N_{\mc B}}f(x)>\lambda\right)\leq \frac{AC_{\alpha}}{\lambda}\|f\|_1+\left(1+\frac{AC_{\alpha}}{\lambda}\|f\|_1\right)\varepsilon.$$
On letting $\varepsilon\to 0$, we conclude that ${N_{\mc B}}$ is of weak type $(1,1)$. Since ${N_{\mc B}}$ is  bounded from $\mathrm{L}^{\infty}$ to $\mathrm{L}^{\infty}$, by the Marcinkiewicz Interpolation Theorem \cite[Theorem 1.3.2]{Grafakos}, ${N_{\mc B}}$ is of strong type $(p,p)$ for $p >1$.

\subsection{Proof of Corollary \ref{cor:cont1}}

The strategy to prove an ergodic theorem given a maximal inequality is standard (see for example \cite[Section 2.6.5]{EW11}). We therefore only sketch the proof. Let us start with an observation.

\begin{rmk}
\label{rmk:lkitoinfty}
If for some $1\leq i\leq d$ Condition \eqref{eq:allomegacont} is satisfied, then it must be $s_{ki}\to\infty$. In fact, if there is a sequence $k_r$ such that $s_{k_ri}\leq C$ for a given constant $C\geq 1$, then for all integer $\lambda\geq C$ and all $r$ we have that
$$\{(w_{k_r i}+x,\lambda):-\alpha\cdot(\lambda-C)\leq x\leq \alpha\cdot(\lambda-C)\}\subset\Omega_{i}^{(\alpha)}(\lambda),$$
whence $\mathrm{Leb}\left(\Omega_{i}^{(\alpha)}(\lambda)\right)=\infty$.
\end{rmk}

The first step in the proof is the following lemma.

\begin{lem}
\label{lem:ltwo}
Let $f\in \Ltwo(X)$. Then there exists a function $f'\in \Ltwo(X)$ that is invariant under the flow $U_{\bs t}$, such that
$$\|R_kf-f'\|_{2}\to 0.$$
\end{lem}

\begin{proof}[Sketch of Proof]
For $i=1,\dotsc,d$ define $\mc U_{U_{t_i}}f:=f\circ U_{t_i}$ and let
$$I:=\{g\in \Ltwo(X):\mc U_{U_{t_i}}g=g\mbox{ for }i=1,\dotsc,d\}.$$    
Let $B:=I^{\perp}$. Then
$$B=\overline{B_{1}\oplus\dotsb\oplus B_d},$$
where
$$B_{i}:=\{\mc U_{U_{t_i}}g-g:g\in\Ltwo(X)\}.$$
Since $s_{ki}\to\infty$ for $i=1,\dotsc,d$ (see Remark \ref{rmk:lkitoinfty}) it is clear that for any $h\in B_i$ we have that $\|R_kh\|_{2}\to 0$ as $k\to\infty$. Then the conclusion follows as in \cite[Theorem 2.21]{EW11}.
\end{proof}

Let $f\in \textup{L}^{\infty}(X)$. By Lemma \ref{lem:ltwo}, we know that $R_kf$ converges to an invariant function $f'$ in $\Ltwo(X)$. Now, for every measurable $B\subset X$ we have that
$$\langle R_{k}f,\chi_{B}\rangle\leq \|f\|_{\infty}\cdot \mu(B).$$
Hence, the same must be true if $R_kf$ is replaced by $f'$. This shows that $f'\in \textup{L}^{\infty}(X)$. 

We now need the following.

\begin{lem}
\label{lem:comp}
For any $k\geq h$ and any $f\in\textup{L}^{\infty}(X)$ we have that
$$R_{k}\circ R_hf=R_k f+O_{\bs w_h,\bs s_h}\left((s_{k1}\dotsm s_{kd})^{-1}\|f\|_{\infty}\right).$$
\end{lem}

\begin{proof}
In this proof, we once again denote by $U_{t}^{(i)}$ the one-dimensional flow $U_{t\bs e_i}$ for $i=1,\dotsc,d$. We have that
\begin{align*}
R_{k}\circ R_hf & =\frac{1}{s_{k1}\dotsm s_{kd}}\int_{t_1'=0}^{t_1'=s_{k1}}\dotsb \int_{t_d'=0}^{t_d'=s_{kd}} \\
& \hspace{0.5cm}\frac{1}{s_{h1}\dotsm s_{hd}}\int_{t_1=0}^{t_1=s_{h1}}\dotsb \int_{t_d=0}^{t_d=s_{hd}}f\left(U^{(1)}_{w_{h1}+w_{k1}+t_1+t_1'}\dotsb U^{(d)}_{w_{hd}+w_{kd}+t_d+t_d'}x\right)\,\dd \bs t\,\dd\bs t' \\
& = \frac{1}{s_{h1}\dotsm s_{hd}}\int_{t_1=0}^{t_1=s_{h1}}\dotsb \int_{t_d=0}^{t_d=s_{hd}} \\
& \hspace{0.5cm}\frac{1}{s_{k1}\dotsm s_{kd}}\int_{t_1'=w_{h1}+t_1}^{t'_1=s_{k1}+w_{h1}+t_1}\dotsb \int_{t_d'=w_{hd}+t_d}^{t_d'=s_{kd}+w_{hd}+t_d}f\left(U^{(1)}_{w_{k1}+t_1'}\dotsb U^{(d)}_{w_{kd}+t_d'}x\right)\,\dd \bs t'\,\dd\bs t \\[2mm]
& = R_{k}f +O_{\bs w_h,\bs s_h}\left((s_{k1}\dotsm s_{kd})^{-1}\|f\|_{\infty}\right),
\end{align*} as desired.
\end{proof}

Let us show point-wise convergence in $\textup{L}^{\infty}(X)$. Let $f\in \textup{L}^{\infty}(X)$ and let $f'$ be the function found in Lemma \ref{lem:ltwo}. Note that for any $k$ it holds that $R_kf'=f'$, since $f'$ is invariant. Fix $\varepsilon,\delta>0$ and pick $h$ so large that $\|R_hf-f'\|_2\leq \delta$. Then, by Lemma \ref{lem:comp}, Remark \ref{rmk:lkitoinfty}, and Theorem \ref{thm:cont1}, we have that
\begin{multline*}
\mu\left(x:\limsup_{k}|R_{k}f-f'|>\varepsilon\right)=\mu\left(x:\limsup_{k}|R_{k}\circ R_h f-R_k f'|>\varepsilon\right) \\
\leq \mu\left(x:{N_{\mc B}}(R_hf-f')>\varepsilon\right)\ll \|R_hf-f'\|_2\leq \delta.
\end{multline*}
This gives point-wise convergence. Finally, since $\textup{L}^{\infty}(X)$ is dense in $\textup{L}^p(X)$ for any $p>1$, we may use Theorem \ref{thm:cont1} and a density argument to conclude. 

\section{Proof of Theorem \ref{thm:cont2}}

In this section we show that, provided Condition \eqref{eq:ctslinearfail} holds for some $1\leq i\leq d$, the operators  $R_k$ enjoy the "strong sweeping out" property. The proof is similar to that of \cite{JONES1991127}[Theorem 3.2] and is a direct application of \cite[Theorem 3]{BJR90}, which we recall below for the convenience of the reader.

\begin{theorem}{\cite[Theorem 3]{BJR90}}
\label{thm:3}
    Let $(X, \Sigma, \mu)$ be a probability space and let $\{T_k\}$ be a sequence of linear operators on $\Lone(X)$ satisfying the following properties:
    \begin{itemize}
        \item $T_k \geq 0$;\vspace{2mm} 
        \item $T_k1=1$;\vspace{2mm}
        \item all $T_k$ commute with a mixing family of measure-preserving transformations $\{S_{h}\}$ on $X$.
    \end{itemize}
    For $n \in \mb N$ let $M_n f:=\sup_{k \geqslant n}|T_kf|$ and assume that for each $\varepsilon>0$ and $n \in \mb N$ there exists a sequence of sets $\{H_p\}$ in $X$ such that \begin{equation}\label{eq:star}\tag{$\star$}
    \sup_p \frac{\mu(M_n\chi_{H_p}>1- \varepsilon)}{\mu(H_p)}= \infty.
\end{equation} Then the sequence $T_k$ has the "strong sweeping out" property (see Theorem \ref{thm:cont2}).
\end{theorem}

When $T_k=R_k$, the first two assumptions in Theorem \ref{thm:3} are trivially true, while the third one follows by our hypothesis in Theorem \ref{thm:cont2}. Hence, it is enough for us to verify that \eqref{eq:star} holds. Note also that, when $f$ is the characteristic function of a set in $X$, $M_nf$ coincides with  $N_{\mc B'}f$, where $\mc B'=\{B_k\in\mc B:k\geq n\}$.

We first need a Rokhlin tower-type construction for aperiodic flows proved in \cite{Lind}, which we recall below.
    \begin{theorem}{\cite[Theorem 1]{Lind}}
    \label{thm:contRokh}
        Let $U_{\bs t}=U_{t_1, \ldots, t_d}$ be a $d$-dimensional measure-preserving aperiodic flow on $(X, \mu)$. Let $L_1, \ldots, L_d,\delta>0$ and let $Q=Q_L:=[0, L_1) \times \ldots \times [0, L_d)$. Then there exists a set $B \subset X$ with the following properties:
    \begin{itemize}
        \item the sets $U_{\bs t }B$ for $\bs t \in Q$ are pairwise disjoint;\vspace{2mm} 
        \item the set $Y:= \bigcup_{\bs t \in Q} U_{\bs t} B$ is measurable and $\mu (Y)> 1- \delta$;\vspace{2mm}
        \item there exists a measure $\nu _B$ defined on $B$ such that the map $\varphi: B \times Q \to X$ given by $\varphi(x,\bs t):= U_{\bs t}x$ is bijective and both $\varphi$ and its inverse are measurable and measure-preserving with respect to the measures $\nu_B\otimes\textup{Leb}$ on $B\times Q$ and $\mu$ on $X$. 
    \end{itemize}
    \end{theorem}
    
    In particular, the last part of Theorem \ref{thm:contRokh} implies that for any $f\in\Lone(X)$ we have that
    \begin{equation}\label{eq:part3}
        \int_Y f \, \dd \mu = \int_B\int_{\bs t \in Q} f(U_{\bs t} x) \, \dd \bs t \dd \nu_B (x).
    \end{equation} 
    
    Take $\alpha=1$ and suppose that \eqref{eq:ctslinearfail} holds for $i=1$. Fix $p \in \mb N$ and choose a real number $\lambda_p>0$ so that $\mathrm{Leb}\left(\Omega_1^{(1)}(\lambda_p)\right)\geq p\cdot 4\lambda_p$. Then for each $z \in \Omega_1^{(1)}(\lambda_p)$ we have that
    \begin{equation}
    \label{eq:inconecont}
    |z-w_{k1}| \leq \lambda_p-s_{k1}    
    \end{equation}
    for some $(w_{k1}, s_{k1}) \in \Omega_1$. Once again, this implies that
    $$\Omega_1^{(1)}(\lambda_p)=\bigcup_{k}C_{k}^{(1)}(\lambda_p),$$
    where
    $$C_{k}^{(1)}(\lambda_p):=\{x\in \mb R:|x-w_{k1}|\leq(\lambda_p-s_{k1})\}.$$
    Choose $K_p\geq 1$ large enough so that
    $$\textup{Leb}\left(\bigcup_{k\leq K_p}C_k^{(1)}(\lambda_p)\right)\geq p\cdot 4\lambda_p$$
    and put
    $$\Delta:=\bigcup_{k\leq K_p}C_k^{(1)}(\lambda_p).$$
    Then for each $z\in \Delta$ there exists $k\leq K_{p}$ such that \eqref{eq:inconecont} holds. Note that it must be $\lambda_p \geq s_{k1}$ so that for all $0\leq t_1<s_{k1}$ we have that
    \begin{equation}\label{eq:existswksk}
        |-z+w_{k1}+t_1|\leq 2 \lambda_p.
    \end{equation}
    Let
    $$L_1 := 2\lambda_p+\sup\Delta$$
    and
    $$L_j := \sup_{k\leq K_p} |w_{kj}+s_{kj}|$$
    for $j=2, \ldots, d$. Form a tower as in Theorem \ref{thm:contRokh} with parameters $L_1, 3L_2, \dotsc,3L_d$.
    Define
    $$Q_p:=[L_1-4 \lambda_p, L_1)\times[0, 3L_2)\times\dotsm\times[0,3L_d)\vspace{2mm}$$
    and
    $H_p:= \{U_{\bs t} x : x \in B \text{ and }\bs t \in Q_p\}$, so that
    $$\chi_{H_p}(y)= \begin{cases}
        \chi_{Q_p} (\bs t)  &  \text{if }y= U_{\bs t} x \text{ with } x\in B \text{ and }\bs t \in Q. \\
        0 & \text{otherwise.}
    \end{cases}$$ By \eqref{eq:part3}, we deduce that
    \begin{multline}\label{eq:H_p}
        \mu(H_p) = \int_{Y} \chi_{H_p} (y) \, \dd \mu =\int_B\int_{\bs t \in Q} \chi_{H_p}(U_{\bs t}x) \, \dd \bs t \dd \nu_B(x) \\
        = \int_B\int_{\bs t \in Q} \chi_{Q_p} (\bs t) \, \dd \bs t \dd \nu_B(x) = \frac{\mathrm{Leb}(Q_p)}{L_1\cdot 3L_2\dotsm 3L_d} \cdot \mu(Y) = \frac{4 \lambda_p}{L_1}\cdot \mu(Y).
    \end{multline} 
     Now, let $$F_p:=\{U_{\bs t}x:x\in B\mbox{ and }\bs t\in\{L_1-2\lambda_p\}\times[L_2,2L_2)\times\dotsb\times[L_d,2L_d)\}.$$
     Fix $z \in \Delta$ and $y \in F_p$, so that $y=U_{\bs a}x$ with $x\in B$, $a_1=L_1-2\lambda_p$, and $a_j\in[0,L_j)$ for $j=2,\dotsc,d$.
     By \eqref{eq:existswksk}, there exists $k\leq K_p$ such that
     $$|a_1-z+w_{k1}+t_1|\in[L_1-4\lambda_p,L_1)\mbox{ for all }t_1 \in [0,s_{k1}).$$
     Since $|w_{kj}+s_{kj}|\leq L_j$ for $j=2,\dotsc,d $, we conclude that
    \begin{multline*}
        N_{\mathcal{B}}\chi_{H_p}(U_{-z,0,\dotsc, 0} y)=N_{\mathcal{B}}\chi_{H_p}(U_{-z,0, \dotsc, 0 }U_{\bs a } x) \\
        \geq  \frac{1}{s_{k1}\ldots s_{kd}} \int_{\bs t \in B_k} \chi_{{H}_p}(U_{-z + a_1+ t_1 , a_2+t_2, \ldots, a_d + t_d}x) \, \dd \bs t= 1.
    \end{multline*}
   This implies that
    $$\left\{{N_{\mc B}}\chi_{H_p}>1- \varepsilon\right\}\supset \bigcup_{z\in \Delta}U_{-z,0,\dotsc, 0}F_p.$$
    Thus, precisely as in \eqref{eq:H_p}, we have that
    $$\mu\left(\bigcup_{z\in \Delta}U_{-z,0,\dotsc, 0}F_p\right)=\frac{\textup{Leb}\left(\Delta\right)\cdot L_2\dotsm L_d}{L_1\cdot 3L_2\dotsm 3L_d}\cdot\mu(Y)\geq \frac{p\cdot 4\lambda_p}{3^{d-1}\cdot L_1}\cdot\mu(Y),$$
    whence
     $$\frac{\mu\left(N_{\mathcal{B}}\chi_{H_p}>1- \varepsilon\right)}{\mu( H_p) } \geq \frac{p\cdot 4\lambda_p}{3^{d-1}\cdot L_1} \cdot \frac{L_1}{4\lambda_p}= \frac{p}{3^{d-1}},$$
     showing \eqref{eq:star} for the operator ${N_{\mc B}}$. To conclude, note that if $\mathcal{B}':=\{B_k\in\mc B:k\geq n\}$, then \eqref{eq:ctslinearfail} holds for the collection $\mc B'$, and the above argument also applies to the operator $N_{\mc B'}$.

\section{Proof of Theorem \ref{thm:cont3}}
    
Choose vectors $\bs u,\bs v_1,\dotsc,\bs v_m\in\mb R^d$ such that
$$U:=\{\bs u+\lambda_1\bs v_1+\dotsb +\lambda_m\bs v_m:\lambda_1,\dotsc,\lambda_m\in(0,1)\}\subset M\cap\pi$$
is an open set in $M$. Since $\pi$ does not contain the origin, we may always assume that $\bs u$ and $\bs v_1,\dotsc,\bs v_m$ are linearly independent. Then for any measurable set $E\subset X$, any $x\in X$, and any $t>0$ we have that
\begin{equation}
\label{eq:1}
\int_{tM}\chi_E(\bs a.x)\,\dd \vol_m(\bs a)\geq \int_{tU}\chi_E(\bs a.x)\,\dd\vol_m(\bs a).    
\end{equation}
Let us study the integral at the right-hand side. Consider the parametrization of $tU$ given by
$$\varphi_t(\bs\lambda)=t\bs u+t\lambda_1\bs v_1+\dotsb +t\lambda_m\bs v_m.$$
Then, if $V:=(\bs v_1,\dotsc,\bs v_m)$, we have that
\begin{equation}
\label{eq:2}
\int_{tU}\chi_E(\bs a.x)\,\dd\vol_m(\bs a)=\int_{(0,1)^m}\chi_{E}(\varphi_t(\bs\lambda).x)\cdot t^m\sqrt{\det\left(V^{T}V\right)}\,\dd \bs\lambda.    
\end{equation}
Note that, since the vectors $\bs v_1,\dotsc,\bs v_m$ are linearly independent, the determinant is non-null. By combining \eqref{eq:1} and \eqref{eq:2}, we deduce that
\begin{equation}
\label{eq:2.25}
\frac{1}{t^m\cdot \vol_m(M)}\int_{tM}\chi_E(\bs a.x)\,\dd\vol_m(\bs a)\geq \frac{\sqrt{\det\left(V^{T}V\right)}}{\vol_{m}(M)}\int_{(0,1)^m}\chi_{E}(\varphi_t(\bs\lambda).x)\,\dd\bs \lambda. 
\end{equation}

Now, let us consider a new action of $\mb R^{m+1}$ on $X$ (which we denote by $``.."$), defined by the relations:
$$\bs e_0..x:=\bs u.x\quad\mbox{and}\quad \bs e_i..x:=\bs v_i.x\mbox{ for }i=1,\dotsc,m.$$
Then for $\bs\lambda=(\lambda_1,\dotsc,\lambda_m)$ we have that
\begin{multline}
\label{eq:2.5}
\int_{(0,1)^m}\chi_{E}(\varphi(\bs\lambda).x)\,\dd\bs \lambda=\int_{(0,1)^m}\chi_{E}\Big((t\bs e_0+t\lambda_1\bs e_1+\dotsb+ t\lambda_m\bs e_m)..x\Big)\,\dd\bs \lambda \\
=\frac{1}{t^m}\int_{(0,t)^m}\chi_{E}\Big((t\bs e_0+\lambda_1\bs e_1+\dotsb+ \lambda_m\bs e_m)..x\Big)\,\dd\bs \lambda.
\end{multline}
For $k\in\mb N$ let
$$B_k:=[k-1,k)\times [0,k)^{m}.$$
Note that, in the notation of Theorem \ref{thm:cont2}, we have that $s_{k1}=1$ for all $k$. Thus, by Remark \ref{rmk:lkitoinfty}, Condition \eqref{eq:ctslinearfail} holds for $i=1$. Moreover, the action of $\mb R^{m+1}$ defined above is aperiodic, since $\bs u,\bs v_1,\dotsc,\bs v_m$ are linearly independent, and it commutes with a mixing family of transformations on $X$, by Remark \ref{rmk:subgrab}.  Thus, by Theorem \ref{thm:cont2}, for any $\varepsilon>0$ there exist a measurable set $E_{\varepsilon}\subset X$ and a sequence of integers $k_r$ such that for all $r$ and $\mu$-a.e. $x\in X$ it holds that
$$\frac{1}{k_r^m}\int_{[k_r-1,k_r)\times [0,k_r)^m}\chi_{E_{\varepsilon}}\Big((\lambda_0\bs e_0+\lambda_1\bs e_1+\dotsb +\lambda_m\bs e_m)..x\Big)\,\dd \lambda_0\dd\bs\lambda>0.99.$$
By Fubini's Theorem, there must be a real number $t_r\in[k_{r}-1,k_r)$ such that
$$\frac{1}{k_r^m}\int_{[0,k_r)^m}\chi_{E_{\varepsilon}}\Big((t_r\bs e_0+\lambda_1\bs e_1+\dotsb+\lambda_m\bs e_m)..x\Big)\dd \bs\lambda> 0.99,$$
whence
\begin{multline*}
\frac{1}{k_r^m}\int_{(0,t_r)^m}\chi_{E_{\varepsilon}}\Big((t_r\bs e_0+\lambda_1\bs e_1+\dotsb+\lambda_m\bs e_m)..x\Big)\dd \bs\lambda \\
=\frac{1}{k_r^m}\int_{[0,k_r)^m}\chi_{E_{\varepsilon}}\Big((t_r\bs e_0+\lambda_1\bs e_1+\dotsb+\lambda_m\bs e_m)..x\Big)\dd \bs\lambda+O_m(k_r^{-1})>0.99+O_m(k_r^{-1}).
\end{multline*}
From this, we deduce that
\begin{multline}
\label{eq:3}
\frac{1}{t_r^m}\int_{(0,t_r)^m}\chi_{E_{\varepsilon}}\Big((t_r\bs e_0+\lambda_1\bs e_1+\dotsb \lambda_m\bs e_m)..x\Big)\dd \bs\lambda\\>0.99\cdot\frac{k_{r}^m}{t_r^m}+O_m(t_r^{-1})=0.99+O_m(t_r^{-1}).    
\end{multline}
Combining \eqref{eq:2.25}, \eqref{eq:2.5}, and \eqref{eq:3}, we conclude that for all $r$ it holds that 
$$\frac{1}{t_r^m\cdot \vol_m(M)}\int_{t_rM}\chi_E(\bs a.x)\,\dd \vol_m(\bs a)\geq 0.99\cdot \frac{\sqrt{\det\left(V^{T}V\right)}}{\vol_{m}(M)}+O_{m,M,V}(t_r^{-1}).$$
This contradicts the convergence to $\mu(\chi_{E_{\varepsilon}})$ if $\varepsilon$ is sufficiently small.

\bibliographystyle{alpha}
\bibliography{References}

\end{document}